\documentclass{conm-p-l}
\usepackage{amsthm,amssymb,amsmath,array,url,booktabs}
\usepackage[dvips]{graphicx}
\usepackage[dvips,all]{xy}
\usepackage[colorlinks,citecolor=black,linkcolor=black,urlcolor=black,bookmarks=false,implicit=false]{hyperref}
\def\arXiv#1{arXiv:\href{http://arXiv.org/abs/#1}{#1}}
\def\MR#1#2{\href{http://www.ams.org/mathscinet-getitem?mr=#1}{MR#1}}
\def\doi#1{\href{http://dx.doi.org/#1}{DOI #1}}

\pagespan{131}{141}

\copyrightinfo{2018}{American Mathematical Society}

\newtheorem{theorem}{Theorem}

\newtheorem{lemma}[theorem]{Lemma}
\newtheorem{proposition}[theorem]{Proposition}
\numberwithin{theorem}{section}

\theoremstyle{definition}

\newtheorem{example}[theorem]{Example}

\theoremstyle{remark}
\newtheorem{remark}[theorem]{Remark}

\numberwithin{equation}{section}

\makeatletter
\def\theenumi{\@roman\c@enumi}
\makeatother

\newcommand{\Z}{\mathbb{Z}}
\newcommand{\Q}{{\mathbb Q}}
\newcommand{\C}{{\mathbb C}}
\newcommand{\F}{{\mathbb F}}
\newcommand{\Qbar}{{\kern.1ex\overline{\kern-.1ex\Q\kern-.1ex}\kern.1ex}}
\newcommand{\E}{{\mathcal E}}

\let\P\Proj
\newcommand{\End}{\operatorname{End}}
\newcommand{\NS}{\operatorname{NS}}
\newcommand{\Km}{\mathit{Km}}
\newcommand{\rank}{\operatorname{rank}}
\newcommand{\II}{\mathrm{II}}
\newcommand{\III}{\mathrm{III}}
\newcommand{\I}{\mathrm{I}}
\newcommand{\IV}{\mathrm{IV}}
\newcommand{\kbar}{{\bar k}}
\newcommand{\sM}{\mathcal{M}}
\newcommand{\sA}{\mathcal{A}}
\DeclareMathOperator{\MW}{MW}
\DeclareMathOperator{\Frob}{Frob}
\def\K3{$K3$}

\title[Inose's construction and elliptic \K3 surfaces with rank~$15$]
{Inose's construction and elliptic \K3 surfaces with Mordell-Weil rank~$15$ revisited}

\author{Abhinav Kumar} 
\address{Department of Mathematics, Stony Brook University, Stony Brook, NY 11794}
\email{thenav@gmail.com}

\author{Masato Kuwata}
\address{Faculty of Economics, Chuo University, 742-1 Higashinakano, 
Hachioji-shi, Tokyo 192-0393, Japan}
\email{kuwata@tamacc.chuo-u.ac.jp}


\date{May 18, 2018}
	

\begin{document}

\begin{abstract}
We describe two constructions of elliptic \K3 surfaces starting from the Kummer surface of the Jacobian of a genus $2$ curve. These parallel the base-change constructions of Kuwata for the Kummer surface of a product of two elliptic curves. One of these also involves the analogue of an Inose fibration. We use these methods to provide explicit examples of elliptic \K3 surfaces over the rationals of geometric Mordell-Weil rank $15$.
\end{abstract}

\maketitle

\section{Introduction}

Elliptic \K3 surfaces have been the focus of much work in algebraic geometry and number theory over the last few decades. For such a surface over a field of characteristic~$0$ with at least one singular fiber, the rank of the Mordell-Weil group can be as large as~$18$. Cox \cite{Cox} proved in 1989, using the surjectivity of the period map for \K3 surfaces, that any integer between~$0$ and~$18$  actually occurs as the rank of the Mordell-Weil group of an elliptic \K3 surface over $\C$. However, the proof uses transcendental methods, and does not lead to an algebraic construction of examples of elliptic \K3 surfaces of a given rank. 

In 2001 Kuwata constructed explicit examples of elliptic \K3 surfaces over $\Q$ whose Mordell-Weil group over $\Qbar$ has rank~$r$, where $0\le r\le 18$ and $r\neq 15$ (\cite{Kuwata:MW-rank}). The proof uses Inose's theorem (see Proposition~\ref{prop:Inose}), which relies on a transcendental argument.  For this reason, it was not clear how to find explicit generators of the Mordell-Weil group of such elliptic surfaces until our recent work \cite{Kumar-Kuwata:singular}, where we found a systematic method to describe a set of generators (see also \cite{Kloosterman-explicit} for earlier progress on this question).  Meanwhile, there have been constructions of elliptic \K3 surfaces with Mordell-Weil rank~$15$ (\cite{Kloosterman-rank15}\cite{Top-de Zeeuw}) in order to fill the gap of~\cite{Kuwata:MW-rank}.  However, these constructions are based on somewhat different ideas from that of \cite{Kuwata:MW-rank}, to which the method of \cite{Kumar-Kuwata:singular} does not seem to apply to find generators of the Mordell-Weil group.

The goal of this article is to describe some new constructions of elliptic \K3 surfaces with high Mordell-Weil rank using ideas closely related to the idea of \cite{Kuwata:MW-rank}. The construction of \cite{Kuwata:MW-rank} is based on the Kummer surfaces $\Km(E_{1}\times E_{2})$ associated with the product of two elliptic curves. For such a Kummer surface Shioda and Inose \cite{Shioda-Inose} constructed a double cover with certain properties (now called a Shioda-Inose structure). We say that a \K3 surface $S$ has a Shioda-Inose structure if $S$ admits an involution fixing the global $2$-form on $S$, such that the quotient is a Kummer surface $\Km(A)$, and provided that the rational quotient map $S \dashrightarrow \Km(A)$ induces a Hodge isometry $T_{S}(2) \simeq T_{\Km(A)}$ of (scaled) transcendental lattices. When $A=E_{1}\times E_{2}$ is the product of two elliptic curves, $S$ was constructed in \cite{Shioda-Inose} as an elliptic surface having two $\II^*$ fibers. Shioda and Inose used this construction to prove a beautiful theorem establishing a one-to-one correspondence between singular \K3 surfaces and even positive definite binary quadratic forms.  Later, Inose \cite{Inose:quartic} constructed $S$ explicitly as a quartic surface in $\P^{3}$ and as the quotient of a Kummer surface by an involution.  Thus, we have a chain of rational maps
\(S\dashrightarrow \Km(A)\dashrightarrow S\),
which we now call a ``Kummer sandwich'' (\cite{Shioda:Kummer-sandwich}). The construction of $\Km(E_1 \times E_2)\dashrightarrow S$ in \cite{Inose:quartic} can be viewed as a base change of elliptic surfaces, and it is this fact that \cite{Kuwata:MW-rank} generalized to construct elliptic \K3 surfaces of high rank. Elliptic \K3 surfaces of various Mordell-Weil ranks are constructed as a base change from $S$, which we call the Inose surface.

Our main idea in this article is to replace $A=E_{1}\times E_{2}$ by the Jacobian $J(C)$ of a curve $C$ of genus~$2$. We propose two different generalizations of Inose's fibration on $\Km(E_{1}\times E_{2})$. A detailed study of Shioda-Inose structures on \K3 surfaces was carried out by Morrison \cite{Morrison}. For Kummer surfaces of principally polarized abelian surfaces, double covers with Shioda-Inose structure were studied by Naruki \cite{Naruki} and Dolgachev \cite{Galluzzi-Lombardo-Dolgachev}. In \cite{Kumar:2008}, Kumar explicitly described these surfaces as a family of elliptic \K3 surfaces with $\II^*$ and $\III^*$ fibers, whose Weierstrass coefficients are related to invariants of the genus $2$ curve $C$. In this construction, the rational map $S\dashrightarrow \Km(A)$ in \cite{Kumar:2008} is a 2-isogeny between elliptic surfaces. The dual isogeny sets up the sandwich $S\dashrightarrow \Km(A)\dashrightarrow S$. Our first construction is based on the observation that $\Km(A)\dashrightarrow S$ can also be described as an Inose-style base change. More precisely, there is an elliptic fibration on the Kummer surface (fibration~13 in \cite{Kumar:2014}) which can be obtained from the Shioda-Inose surface in \cite{Kumar:2008} as a base change by a simple change of variables. In \S3 we give the details of this construction of elliptic \K3 surfaces of various Mordell-Weil ranks. In particular, we give two numerical examples that have Mordell-Weil rank~$15$.
 
Note that fibration~13 in \cite{Kumar:2014} has one $\IV^*$ fiber and one $\I_0^*$ fiber, while Inose's fibration on $\Km(E_{1}\times E_{2})$ has two $\IV^*$ fibers.  There does exist an elliptic fibration on $\Km(J(C))$ with two $\IV^*$ fibers, and it can be constructed from known fibrations in \cite{Kumar:2014} using the ``2-neighbor method'' of Elkies.  However, this fibration does not generically have a section. Nevertheless, taking its Jacobian fibration produces an elliptic fibration with section and the same fiber type. Taking a base change, we can construct elliptic \K3 surfaces with Mordell-Weil rank up to $18$.  In \S4 we give details of this construction, and give a numerical example with Mordell-Weil rank~$15$. 

In future work, we hope to find explicit generators for the Mordell-Weil groups of the elliptic \K3 surfaces constructed here, along the lines of \cite{Kumar-Kuwata:singular}.

\subsection{Acknowledgements}
We thank Matthias Sch\"utt and the referees for helpful comments on an
earlier version of the paper. Kumar was supported in part by NSF
CAREER grant DMS-0952486, and by a grant from the MIT Solomon
Buchsbaum Research Fund.  Kuwata was partially supported by JSPS
KAKENHI Grant Number JP26400023, and by the Chuo University Grant for
Special Research.  The computer algebra systems \texttt{Magma},
\texttt{sage}, \texttt{gp/PARI}, \texttt{Maxima} and \texttt{Maple}
were used in the calculations for this paper.

\section{Preliminaries}
By an elliptic surface we mean a smooth projective surface~$S$ together with a relatively minimal elliptic fibration $\pi:S\to C$ over a smooth projective curve $C$, all defined over a field~$k$, which we assume to be a number field throughout. We are interested in the case where $S$ is a \K3 surface, $C$ is isomorphic to $\P^{1}$ over $k$, and $\pi$ admits a section $\sigma_{0}:\P^{1}\to S$ also defined over $k$. Under these assumptions the generic fiber $\E$ of $\pi$ is an elliptic curve defined over the rational function field $k(t)$. We also assume there is at least one singular fiber (to ensure that the Mordell-Weil group is finitely generated). Denote by $\MW_{\kbar}(\E)$ the Mordell-Weil lattice $\E(\kbar(t))/\E(\kbar(t))_{\text{tors}}$. We refer the reader to \cite{Shioda:MWL} for the theory of Mordell-Weil lattices and \cite{Schuett-Shioda} for a broader overview of elliptic surfaces.

A theorem of Shioda and Tate connects the Mordell-Weil $\E(\kbar(t))$ group with the Picard group or the N\'eron-Severi group $\NS_{\kbar}(S)$ of $\E$ (note that linear equivalence and algebraic equivalence coincide). In particular, we have the Shioda-Tate formula (cf.~Shioda\cite{Shioda:elliptic-modular})
\begin{equation}\label{eq:Shioda-Tate}
\rank \NS_{\kbar}(S)= 2 + \rank \MW_{\kbar}(E) +\sum_{P\in C} (m_P-1),
\end{equation}
where $m_P$ is the number of (geometrically) irreducible components of
the fiber $\pi^{-1}(P)$.  We call~$\rho(S) = \rank \NS_{\kbar}(S)$ the (geometric) Picard number.

\begin{remark}
In the absence of a section we would say that $\pi$ is a genus one fibration. Sometimes in the literature these are called elliptic fibrations and those with section are labeled Jacobian elliptic fibrations.
\end{remark}

The following theorem of Inose is essential to our construction.

\begin{proposition}[Inose {\cite[Cor.~1.2]{Inose:quartic}}]\label{prop:Inose}
Let $S_{1}$ and $S_{2}$ be \K3 surfaces defined over a number field $k$ and $f:S_{1}\to S_{2}$ a rational map of finite degree.  Then $S_{1}$ and $S_{2}$ have the same Picard number.
\end{proposition}

\begin{proof}
Inose proves this lemma for \K3 surfaces over $\C$ using a transcendental argument. Since we know that $\NS_{\C}(S)=\NS_{\kbar}(S)$, the conclusion is still valid for our case.
\end{proof}

Due to the transcendental nature of Inose's proof we do not have a very good understanding of how the N\'eron-Severi lattices of $S_{1}$ and $S_{2}$ are related. Thus, knowledge of an explicit set of divisors generating the Picard group of $S_1$, or of the Mordell-Weil group if $S_1$ is elliptically fibered, does not easily translate to the same for $S_2$.

Let $S$ be a Kummer surface associated with an abelian surface~$A$. Then, 
\[
\rank\NS_{\kbar}(S)=\rank\NS_{\kbar}(A)+16.
\]

\begin{proposition}[cf.~Birkenhake-Lang {\cite[Prop.~5.5.7]{Birkenhake-Lange}}]
Let $A$ be a simple abelian surface.  Then, the Picard number $\rho(A)=\rank\NS_{\kbar}(S)$ is given as follows.
\[
\begin{tabular}{cc}
\toprule
$\End_{\Q}(A)$ & $\rho(A)$ \\
\midrule 
$\Q$ & $1$ \\
real quadratic field & $2$ \\
indefinite quaternion algebra & $3$ \\
CM field of degree~$4$ & $2$ \\
\bottomrule
\end{tabular}
\]
\end{proposition}

\begin{proposition}
Let $A$ be an abelian surface isogenous to the product of two elliptic curves $E_{1}\times E_{2}$.  Then, the Picard number $\rho(A)$ is given as follows.
\[
\begin{tabular}{lc}
\toprule
\hfil $A\sim E_{1}\times E_{2}$ & $\rho(A)$ \\
\midrule 
$E_{1}\not\sim E_{2}$ & $2$ \\
$E_{1}\sim E_{2}$, without CM & $3$ \\
$E_{1}\sim E_{2}$, with CM & $4$ \\
\bottomrule
\end{tabular}
\]
\end{proposition}

\section{Inose type surface for the Jacobian of a curve of genus~$2$}

Let $C$ be a curve of genus~$2$ given by the equation
\[
y^2 = f(x) = \sum_{i=0}^6 f_i x^i,
\]
where $f$ is a squarefree polynomial. In particular, we allow $f_{6}$ to be $0$ (but in that case $f_5 \neq 0$). In \cite{Kumar:2008} it is shown that a generic elliptic \K3 surface with $\II^*$ and $\III^*$ fibers has a Shioda-Inose structure such that the quotient is the Kummer surface of the Jacobian of a genus $2$ curve. In fact, this correspondence sets up a birational map between the moduli space of such elliptic \K3 surfaces and $\sM_2$, the moduli space of genus $2$ curves. It can also be viewed as an isomorphism between the moduli space of \K3 surfaces lattice polarized by $U \oplus E_8(-1) \oplus E_7(-1)$ and $\sA_2$, the moduli space of principally polarized abelian surfaces. For a genus $2$ curve $C$, we can describe the unique associated elliptic \K3 surface with $\II^*$ and $\III^*$ fibers as follows.

Let $(I_{2},I_{4},I_{6},I_{8})$ be the Igusa-Clebsch invariants (see for example Section 3.5 of \cite{Kumar:2014}).  Then the corresponding \K3 surface is given by the Weierstrass equation 
\begin{equation}\label{eq:II*-III*}
y^{2} =x^3 - t^{3}\Bigl(\frac{I_{4}}{12}t+1\Bigr)x
+t^{5}\Bigl(\frac{I_{10}}{4}t^2+\frac{I_{2}I_{4}-3I_{6}}{108}t+\frac{I_{2}}{24}\Bigr)
\end{equation}
It has a type $\II^*$ fiber at $t=\infty$ and a type $\III^*$ at $t=0$.  We say that $C$ is ``general'' if the other singular fibers are all $\I_1$.

The surface defined by \eqref{eq:II*-III*} is isomorphic to 
\[
y^{2} =x^3 - \Bigl(\frac{I_{4}}{12}+\frac{1}{t}\Bigr)x
+\Bigl(\frac{I_{10}}{4}t+\frac{I_{2}I_{4}-3I_{6}}{108}+\frac{I_{2}}{24t}\Bigr),
\]
which we call $G^{(1)}$, and define the base change
\[
G^{(n)}:y^{2} =x^3 - \Bigl(\frac{I_{4}}{12}+\frac{1}{t^{n}}\Bigr)x
+\Bigl(\frac{I_{10}}{4}t^{n}+\frac{I_{2}I_{4}-3I_{6}}{108}+\frac{I_{2}}{24t^{n}}\Bigr).
\]
Compare with the definition of $F^{(n)}$ in \cite{Shioda:Sphere-packings} and \cite{Kumar-Kuwata:singular}. Straightforward calculations show that the Kodaira-N\'eron model of $G^{(n)}$ is a \K3 surface for $n=1,\dots,4$.

In \cite{Kumar:2008} the map $G^{(1)}\to \Km(J(C))$ is given as a $2$-isogeny between two elliptic surfaces.  All Jacobian elliptic fibrations (elliptic fibrations with section) on $\Km(J(C))$ are classified in \cite{Kumar:2014}.  Fibration 13 in \cite{Kumar:2014} is 
\[
y^{2} = x^{3} - 108t^{4}\bigl(48t^{2} + I_{4}\bigr)x 
+ 108t^{4}\bigl(72I_{2}t^{4} + (4I_{4}I_{2} - 12I_{6})t^{2} + 27I_{10}\bigr),
\]
which is isomorphic to $G^{(2)}$ by $(x,y,t)\mapsto \big(t^2x/9,t^3y/27,1/(2t) \big)$.  Thus, $G^{(2)}$ is isomorphic to $\Km(J(C))$, and we have a Kummer sandwich diagram: $G^{(1)}\dashrightarrow G^{(2)}\simeq\Km(J(C))\dashrightarrow G^{(1)}$.

\begin{remark}
The surface $G^{(1)}$ can be realized as a quartic surface just as Inose's suface in \cite{Inose:singular-K3}:
\[
y^{2}zw - x^3z + \Bigl(\frac{I_{4}}{12}w+z\Bigr)xzw
-\Bigl(\frac{I_{10}}{4}w^{2}+\frac{I_{2}}{24}z^{2}\Bigr)w^{2}-\Bigl(\frac{I_{2}I_{4}-3I_{6}}{108}\Bigr)zw^{3}=0.
\]
Therefore we can regard $G^{(1)}$ as a generalization of Inose's surface to the case of Jacobian of the genus~$2$ curve~$C$. This point of view was also studied in \cite{Clingher-Doran}.
\end{remark}

\begin{theorem}\label{thm:G^(n)}
Suppose $C$ is general.  Then the rank of the Mordell-Weil group $G^{(n)}(\Qbar(t))$ is given by the table below.  Here, $\rho=\rho(J(C))$ is the Picard number of $J(C)$.
\[
\begin{tabular}{ccc}
\toprule
 & Singular fibers & Rank \\
\midrule 
$G^{(1)}$  & $\mathrm{II}^*, \mathrm{III}^{*}, 5\I_1$ & $\rho-1$ \\
$G^{(2)}$ & $\IV^*, \I_{0}^{*}, 10 \I_1$ & $4+\rho$ \\
$G^{(3)}$ & $\I_0^*, \mathrm{III}, 15 \I_1$ & $9+\rho$ \\
$G^{(4)}$ & $\IV, 20 \I_1$ & $12+\rho$ \\
\bottomrule
\end{tabular}
\]
\end{theorem}

\begin{proof}
Determination of the types of singular fibers is a straightforward application of Tate's algorithm.  By Proposition~\ref{prop:Inose}, the Picard number of each of the \K3 surfaces $G^{(1)},\dots, G^{(4)}$ is equal to that of $\Km(J(C))$, namely, $16+\rho$. Now, calculation of the rank is a straightforward application of the Shioda-Tate formula \eqref{eq:Shioda-Tate}.
\end{proof}

In order to obtain elliptic \K3 surfaces whose Mordell-Weil rank~$15$, it now suffices to find curves $C$ such that the Picard number of its Jacobian $J(C)$ is~$3$.  We give two different examples below.

\begin{remark}
As shown in \cite{Kumar:2014}, a set of five sections forms a basis of $G^{(2)}(\Qbar(t))$ if $\rho(J(C)) = 1$.  These sections are defined over the field $k(J(C)[2])$.
\end{remark}

\begin{example}[QM case] In the case where the Jacobian $J(C)$ is simple, $\rho(J(C))=3$ if and only if it has quaternionic multiplication.  Writing down an explicit equation of such a curve $C$ is a delicate number theoretical problem. Here, we use an example of \cite{Dieulefait-Rotger} which depends on a result of Hashimoto and Tsunogai \cite{Hashimoto-Tsunogai} (see also \cite{Hashimoto-Murabayashi}). Take the curve in \cite[Lemma~4.5]{Hashimoto-Tsunogai} and let $\sigma = \tau = \sqrt{-3/2}$.  Then we obtain 
\[
C:y^{2}= \Bigl(x^{2}+\frac{7}{2}\Bigr)
\Bigl(\frac{83}{30} {x}^{4}+ 14 x^{3}-{\frac {1519}{30}} x^{2}
+49 x-\frac {1813}{120}\Bigr).
\]
It has quaternionic multiplication by a maximal order in the quaternion algebra of discriminant $6$, which is the smallest possible discriminant for a non-split quaternion algebra.
Calculating Igusa-Clebsh invariants, and simplifying the coefficients by a change of variables, we obtain the elliptic surface
\[
G^{(4)}:Y^{2} = X^3+529200\Bigl(6-\frac{5}{t^{4}}\Bigr)X
-9261000\Bigl(4t^4+20-\frac{3431}{t^{4}}\Bigr).
\]
Simple calculations show that the singular fibers are irreducible except at $t=\infty$ where it has a fiber of type $\mathrm{IV}$.  This shows that the rank of $G^{(4)}(\Qbar(t))$ equals~$15$.

\end{example}
\begin{remark}
If we take the quadratic twist of $G^{(4)}$ by $\Q(\sqrt{2\cdot3\cdot5\cdot7})$, the elliptic surface becomes simpler:
\[
Y^{2}=X^3+12\Bigl(6-\frac{5}{t^{4}}\Bigr)X
-\Bigl(4t^4+20-\frac{3431}{t^{4}}\Bigr).
\]
However, it is possible that the field of definition of the Mordell-Weil group of the original surface may be easier to describe.
\end{remark}

\begin{example}[Split case] As a second example, we consider the case where $J(C)$ is isogenous to the product of two elliptic curves.  Since we want $\rho(J(C))$ to be equal to~$3$, we need to consider the case where $J(C)\sim E\times E$ with $E$ without complex multiplication.

Take two elliptic curves 
\begin{alignat*}{2}
E_{1}&: y^{2}+xy+y=x^{3}+4x-6 &&(\text{Cremona label 14a1}), \\
E_{2}&: y^{2}+xy+y=x^{3}-36x-70 &\quad &(\text{Cremona label 14a2}),
\end{alignat*}
These both have rational $6$-torsion points, and are $2$-isogenous to each other.  Moreover, we have $E_{1}[3]\simeq E_{2}[3]\simeq\Z/3\Z\times\mu_{3}$ as Galois modules. Take the quadratic twist $E_{2}^{(-3)}$ of $E_{2}$ by $-3$.  Then, the Galois isomorphism $E_{1}[3]\to E_{2}[3]$ isometric with respect to the Weil pairing induces an anti-isometry $E_{1}[3]\to E_{2}^{(-3)}[3]$. Therefore, by Frey-Kani \cite{Frey-Kani}, there exists a curve $C$ of genus~$2$ admitting morphisms $C\to E_{1}$ and $C\to E_{2}^{(-3)}$ of degree~$3$. To find an explicit equation, we use Shaska's result \cite[\S3]{Shaska}.  Using the fact that $j(E_{1})=(215/28)^3$ and $j(E_{2})=(1705/98)^3$, we find that the curve $C$ corresponds to the values $\mathfrak{u} = 107553525/12595352$, $\mathfrak{v} = -11619959625/1032401161$ in equation (8) in \cite{Shaska}.  After making some changes of variables, we obtain the following Weierstrass equation for $C$:
\[
C: y^2=-96393(13x+12)(7x-13)(107x^2-273x+252)(56x^2+104x+31).
\]
The maps $\phi_{1}:C\to E_{1}$ and $\phi_{2}:C\to E_{2}^{(-3)}$ are given by
\begin{align*}
\phi_{1}&:(x,y)\mapsto (x_{1},y_{1})=
\\&\qquad
\left(-\frac{156260 x^3+3627 x+86895}{3(7x-13)(56 x^2+104 x+31)},
-\frac{4(13x-6)(214 x^2+273 x+126)y}{9(7x-13)^2(56 x^2+104 x+31)^2}
\right),\\
\phi_{2}&:(x,y)\mapsto (x_{2},y_{2})=
\\&\qquad
\left(\frac{47485 x^3+4173 x^2-211380}{(13x+12)(107 x^2-273 x+252)},
-\frac{(12(7x+26))(14 x^2-52 x+31)y}{(13x+12)^2(107 x^2-273 x+252)^2}
\right),
\end{align*}
where we have used the following Weierstrass forms for $E_{1}$ and $E_{2}^{(-3)}$:
\[\setlength{\arraycolsep}{1pt}\begin{array}{lll}
E_{1}&: &\ y_{1}^{2}=x_{1}^{3}+5805 x_{1}-285714, \\
E_{2}^{(-3)}&:&\ y_{2}^{2}=x_{2}^{3}-5115 x_{2}+115414.
\end{array}\]
The elliptic \K3 surface $G^{(n)}$ corresponding to this $C$ is given by
\begin{multline*}
G^{(n)}:Y^2 = X^3+33\Bigl(2933005-\frac{1126255812}{t^{n}}\Bigr)X
\\
-2\Bigl(28449792t^{n}-8690133815-\frac{274280846290470}{t^{n}}\Bigr).
\end{multline*}
(Here, we have scaled $X$, $Y$, and $t$  differently from the definition of $G^{(n)}$ to clear denominators.)
It is readily verified that the singular fibers other than the $\mathrm{IV}$ fiber at $t=\infty$ are all irreducible.  Since $E_{1}$ and $E_{2}^{(-3)}$ are isogenous over $\Q(\sqrt{-3})$ and do not have complex multiplication, the rank of $G^{(4)}(\Qbar(t))$ is exactly~$15$.

\end{example}

\begin{remark}
The reader may wonder why we did not use a curve $C$ such that $J(C)$ is $(2,2)$-isogenous to the product $E_{1}\times E_{2}$.  The reason is that the resulting elliptic surface $G^{(1)}$ has at least one $\mathrm{I}_{2}$ fiber other than $\mathrm{II}^{*}$ and $\mathrm{III}^{*}$ fibers.  Thus, the rank of $G^{(4)}(\Qbar(t))$ cannot be~$15$.  Similarly, Example~1 in \cite[\S5]{Shaska} yields $G^{(1)}$ with two $\mathrm{I}_{2}$ fibers, and Example~2 yields one with a $\mathrm{I}_{3}$ fiber.
\end{remark}

\section{Fibration with two $\mathrm{IV}^{*}$ fibers}

A Kummer surface $\Km(J(C))$ can be realized as a complete intersection of three quadrics in $\P^{5}$.  It contains two sets of sixteen lines intersecting each other as shown in \cite[p.767, Figure~21]{Griffiths-Harris}. Classically, these sets are labeled {\em nodes} and {\em tropes}; in the model of the Kummer surface as a quartic surface in $\P^3$ with sixteen singular points, the (blown-down) nodes are the singular points, while the tropes correspond to the sixteen planes tangent to the Kummer surface along plane conics (which are the transforms of the tropes).

As indicated in Figure~1, we find a configuration of a pair of divisors of type $\mathrm{IV}^{*}$.  We thus have an elliptic fibration on $\Km(J(C))$ having these two $\mathrm{IV}^{*}$ fibers.  It turns out, however, that such an elliptic fibration does not have a section.  Thus, it is not in the list of elliptic fibrations studied in \cite{Kumar:2014}.

\begin{figure}[t]\label{fig:1}
\includegraphics[scale=1.3]{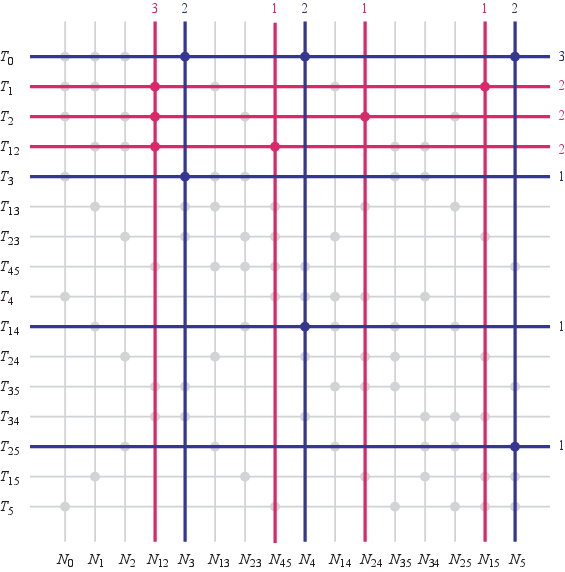}
\caption{Two $\mathrm{IV}^{*}$ fibers in the $16$--$6$ configuration of $\NS(\Km(J(C)))$}
\end{figure}

The elliptic divisor corresponding to the fibration is $F = 3T_0 + 2N_3 + 2N_4 + 2N_5 + T_3 + T_{14} + T_{25}$.

To obtain an explicit equation of this fibration, we start with Fibration~1 in \cite{Kumar:2014}, two $\I_0^*$ fibers and six $\I_2$ fibers, with Mordell-Weil group $\Z \oplus (\Z/2\Z)^2$. The zero section in \cite{Kumar:2014} is taken to be $T_2$. However, we will take $T_3$ to be the zero section; this does not affect the Weierstrass equation, however changes the identity components of the reducible fibers, for instance. We can now take a $2$-neighbor step to the elliptic fibration corresponding to the fiber $F' = T_3 + N_3 + T_0 + N_4 + T_{14} + N_{23}$ (this fiber class is in the orbit of fibration $10$ of \cite{Kumar:2014}). We omit the details of the neighbor step, apart from giving the elliptic parameter
\[
t_1 = \frac{1}{t-a}\left( \frac{y+y_s}{x-x_s} - 2a(b-1)(c-a)  \right)
\]
where
\[
(x_s, y_s) = \big(-4(a-1)(b-1)t(t-c)(ct-ab), -8(a-1)(b-1)(c-a)(c-b)t^2(t-ab)(t-c) \big)
\]
is the equation of the section $T_{14}$. The new elliptic fibration has Weierstrass equation. 
\begin{align*}
y_1^2 &= x_1\big(x_1^2 + x_1(t_1^4/16 + t_1^3(b-1)(c-a)/2 \\
    & \quad    + t_1^2(3b^2c^2-abc^2-4bc^2+2ac^2-4ab^2c-b^2c+2a^2bc \\
    & \quad     +3abc+2bc-4a^2c+2ac+2ab^2+2a^2b-4ab)/2 \\
    & \quad    + 2t_1(b-1)b(b-a)(c-1)c(c-a) + b^2(b-a)^2c^2(c-1)^2 ) \\
    & \quad  - (a-1)a(b-1)(c-a)(c-b)t_1^3(t_1+2bc-2b)(t_1+2bc-2ac)/2 \big)
\end{align*}
The fiber $F'$ is at $\infty$, while there is another fiber of type $\I_6$ at $t_1 = 0$, given by $T_1 + N_{12} + T_{12}+ N_{45} + T_5 + N_{15}$. There are also $\I_2$ fibers at $t_1 = -2(b-a)c$ and $t_1 = -2b(c-1)$. The first is $T_{25} + N'_{35}$ (in the notation of Fibration~$1$ of \cite{Kumar:2014}) and the second has $N_{24}$ as a component. We can now move by another $2$-neighbor step to the elliptic fibration with $F$ as a fiber. The elliptic parameter is
\[
t_2 = \frac{x_1}{t_1 + 2(b-a)c} ;
\]
we omit the remaining details.
The Jacobian of this new fibration is given by 
\[
Y^{2} = X^{3} + A X^{2} 
+ \Bigl(B_{1} T + B_{2} + \frac{B_{3}}{T}\Bigr)X
+\Bigl(C_{1} T + \frac{C_{2}}{T}\Bigr)^{2},
\]
where $T = t_2$ and
\[\setlength{\arraycolsep}{2pt}
\begin{array}{ccl}
A &= &4\bigl(abc^2+bc^2-2ac^2+ab^2c-2b^2c
-2a^2bc+bc+4a^2c-2ac+ab^2 -2a^2b\\
& & \quad +ab\bigr),\\
B_{1} &= &(a-1)c-a(b-1),\\
B_{2} &= &-16\bigl(ab^2c^4-a^2bc^4-abc^4+a^2c^4+a^2b^3c^3-ab^3c^3-b^3c^3-a^3b^2c^3\\
&& \qquad +2a^2b^2c^3-2ab^2c^3+b^2c^3+a^2bc^3+2abc^3-2a^2c^3-ab^4c^2+b^4c^2\\
&& \qquad -a^3b^3c^2-2a^2b^3c^2+6ab^3c^2-b^3c^2+a^4b^2c^2+a^3b^2c^2-3a^2b^2c^2\\
&& \qquad -2ab^2c^2+a^2bc^2-abc^2+a^2c^2+a^2b^4c-ab^4c+2a^3b^3c-2a^2b^3c-ab^3c\\
&& \qquad -2a^4b^2c+a^3b^2c+2a^2b^2c+ab^2c-a^2bc-a^3b^3+a^2b^3+a^4b^2-a^3b^2\bigr), \\
B_{3} &=& -ab^2c(a-1)(b-1)(c-1)^2(a-b)(b-c)(c-a)
\bigr((a-1)c-a(b-1)\bigr), \\
C_{1} &=& 1, \\
C_{2} &=& ab^2c(a-1)(b-1)(c-1)^2(a-b)(b-c)(c-a).
\end{array}
\]
It has $\IV^*$ fibers at $T=0$ and $\infty$.

Define
\[
H^{(n)}: Y^{2} = X^{3} + A X^{2} 
+ \Bigl(B_{1} t^{n} + B_{2} + \frac{B_{3}}{t^{n}}\Bigr)X
+\Bigl(C_{1} t^{n} + \frac{C_{2}}{t^{n}}\Bigr)^{2}.
\]
The Kodaira-N\'eron model of $H^{(n)}$ is a \K3 surface for $n=1,2,3$.

\begin{theorem}\label{thm:H^(n)}
Suppose $C$ is general.  Then the rank of the Mordell-Weil group $H^{(n)}(\Qbar(t))$ is given by the table below.  Here, $\rho=\rho(J(C))$ is the Picard number of $J(C)$.
\[
\begin{tabular}{ccc}
\toprule
 & Singular fibers & Rank \\
\midrule 
$H^{(1)}$  & $2\IV^*, 8\I_1$ & $2+\rho$ \\
$H^{(2)}$ & $2\IV, 12 \I_1$ & $10+\rho$ \\
$H^{(3)}$ & $24 \I_1$ & $14+\rho$ \\
\bottomrule
\end{tabular}
\]
\end{theorem}

\begin{proof}
Since $H^{(1)}$ is the Jacobian of an elliptic fibration on $\Km(J(C))$, there is a rational map $H^{(1)}\dashrightarrow \Km(J(C))$ of finite degree.  Thus, by Proposition~\ref{prop:Inose}, the Picard number of $H^{(1)}$ is equal to $16+\rho$.  The rest of the proof is similar to that of Theorem~\ref{thm:G^(n)}.
\end{proof}

Thus, $H^{(3)}$ for any simple abelian surface without extra endomorphisms gives an elliptic \K3 surface of rank~$15$.  However, although a very general abelian surface (in the moduli space of principally polarized abelian surfaces) has this property, and we expect a ``random'' numerical example to have $\rho = 1$, we must verify it for any putative example explicitly. To do so, we first use Lepr\'evost's criterion \cite[Lemme~3.1.2]{Leprevost}: a genus $2$ curve $C$ over $\Q$ has simple Jacobian if there is a prime~$\ell$ of good reduction such that the Galois group of the characteristic polynomial of Frobenius of the reduced curve is the 8-element dihedral group $D_{4}$. We can compute the characteristic polynomial of Frobenius in {\tt magma} or {\tt sage}; the latter has an optimized implementation of the $p$-adic Harvey-Kedlaya algorithm for computing the matrix of Frobenius for a curve over a finite field. 

\begin{example} \label{E6E6example}
Let $a = -1$, $b = 1/7$, and $c = -6/7$.  Then, we can verify that the Jacobian $J(C)$ of the curve
\[
C:y^{2} = x(x-1)(x+1)
\Bigl(x-\frac{1}{7}\Bigr)\Bigl(x+\frac{6}{7}\Bigr)
\]
is simple and $\rho(J(C))=1$ by the lemma below.
After a certain change of variables, we have
\[
H^{(n)}:y^{2}=x^3-\frac{1354}{7}x^2
+\frac{936}{7}\Bigl(t^{n}+\frac{42989}{819}+\frac{4}{t^{n}}\Bigr)x
+\frac{6084}{49}\Bigl(t^{n}-\frac{4}{t^{n}}\Bigr)^2.
\]
The Mordell-Weil rank of $H^{(3)}$ is~$15$.
\end{example}

\begin{lemma}
Let $C$ be the curve of Example \ref{E6E6example}. Then $\rho(J(C)) = 1$.
\end{lemma}

\begin{proof}
The characteristic polynomial of $\Frob_{37}$ and $\Frob_{41}$ on $A = J(C)$ are given by 
\begin{align*}
p_{37} &= x^4 - 4x^3 + 46x^2 - 148x + 1369 \\
p_{41} &= x^4 + 4x^3 + 6x^2 + 164x + 1681 .
\end{align*}
We can directly verify (for instance in {\tt gp/PARI} or{ \tt Magma}) that $p_{37}$ has Galois group $D_4$. Therefore, the Jacobian is simple. Next, since $p_{37}$ and $p_{41}$ are irreducible, it follows by \cite[Theorem 2]{Tate} that the endomorphism rings of the reductions $A \otimes \F_{37}$ and $A \otimes \F_{41}$ are orders in quartic fields defined by $p_{37}$ and $p_{41}$. These fields are easily checked to be linearly disjoint. Since $\End(A) \otimes \Q$ is a subfield of each, it must be $\Q$, and $\End(A)$ must be $\Z$.
\end{proof}


\begin{thebibliography}{Shi4}

\bibitem[BL]{Birkenhake-Lange} Christina Birkenhake and Herbert Lange,
  \textit{Complex Abelian Varieties}, 2nd ed., Grundlehren der
  Mathematischen Wissenschaften [Fundamental Principles of
    Mathematical Sciences], vol. 302, Springer-Verlag, Berlin, 2004.
  \doi{10.1007/978-3-662-06307-1}. \MR{2062673}{(2005c:14001)}

\bibitem[CD]{Clingher-Doran} Adrian Clingher and Charles F.~Doran,
  \textit{Lattice polarized $K3$ surfaces and Siegel modular forms},
  Adv. Math. {\bf 231} (2012), no. 1, 172--212,
  \doi{10.1016/j.aim.2012.05.001}. \MR{2935386}{}

\bibitem[C]{Cox} David A. Cox, \textit{{Mordell-Weil} groups of
  ellipitc curves over {$\C(t)$} with {$p_g=0$ or~$1$}}, Duke
  Math. J. \textbf{49} (1982), no. 3, 677--689,
  \doi{10.1215/S0012-7094-82-04935-3}. \MR{0672502}{(84d:14029)}

\bibitem[DR]{Dieulefait-Rotger} Luis V.~Dieulefait, and Victor Rotger,
  \textit{On abelian surfaces with potential quaternionic
    multiplication}, Bull. Belg. Math. Soc. Simon Stevin {\bf 12}
  (2005), no. 4, 617--624. \MR{2206004}{(2007a:11077)}

\bibitem[E]{Elkies} Noam D.~Elkies, \textit{Elliptic curves of high
  rank over $\Q$ and $\Q(t)$\/}, in preparation.

\bibitem[FK]{Frey-Kani} Gerhard Frey and Ernst Kani, \textit{Curves of
  genus 2 covering elliptic curves and an arithmetical application},
  Arithmetic algebraic geometry (Texel, 1989), Progr. Math., vol. 89,
  Birkh\"auser Boston, Boston, MA, 1991, pp. 153--176,
  \doi{10.1007/978-1-4612-0457-2\_7}. \MR{1085258}{(91k:14014)}

\bibitem[GLD]{Galluzzi-Lombardo-Dolgachev} Federica Galluzzi and
  Guiseppe Lombardo, \textit{Correspondences between $K3$ surfaces\/},
  Michigan Math. J., {\bf 52} (2004), no. 2, 267--277,
  \doi{10.1307/mmj/1091112075}. With an appendix by Igor Dolgachev.
  \MR{2069800}{(2005j:14054)}

\bibitem[GH]{Griffiths-Harris} Phillip Griffiths and Joseph Harris,
  \textit{Principles of Algebraic Geometry}, Wiley-Interscience [John
    Wiley \& Sons], New York, 1978. Pure and Applied Mathematics.
  \doi{10.1002/9781118032527}. \MR{0507725}{(80b:14001)}

\bibitem[HM]{Hashimoto-Murabayashi} Ki-Ichiro Hashimoto, and Naoki
  Murabayashi, \textit{Shimura curves as intersections of Humbert
    surfaces and defining equations of QM-curves of genus two}, Tohoku
  Math. J. (2) {\bf 47} (1995), no. 2, 271--296,
  \doi{10.2748/tmj/1178225596}. \MR{1329525}{(96b:14023)}

\bibitem[HT]{Hashimoto-Tsunogai} Ki-ichiro Hashimoto, and Hiroshi
  Tsunogai, \textit{On the Sato-Tate conjecture for QM-curves of genus
    two}, Math. Comp. {\bf 68} (1999), no. 228, 1649--1662,
  \doi{10.1090/S0025-5718-99-01061-3}. \MR{1627797}{(99m:11072)}

\bibitem[I1]{Inose:quartic} Hiroshi Inose, \textit{On certain {K}ummer
  surfaces which can be realized as non-singular quartic surfaces in
  {$P^{3}$}}, J. Fac. Sci. Univ. Tokyo Sect.  IA Math. \textbf{23}
  (1976), no.~3, 545--560. \MR{0429915}{(55 #2924)}

\bibitem[I2]{Inose:singular-K3} Hiroshi Inose, \textit{Defining
  equations of singular {$K3$} surfaces and a notion of isogeny},
  Proceedings of the {I}nternational {S}ymposium on {A}lgebraic
  {G}eometry ({K}yoto {U}niv., {K}yoto, 1977), Kinokuniya Book
  Store, 1978, 495--502. \MR{0578868}{(81h:14021)}
  
\bibitem[Kl1]{Kloosterman-rank15} Remke Kloosterman, \textit{Elliptic
  $K3$ surfaces with geometric Mordell-Weil rank 15\/},
  Canad.\ Math.\ Bull.\ {\bf 50} (2007), no. 2, 215--226,
  \doi{10.4153/CMB-2007-023-2}. \MR{2317444}{(2008f:14055)}

\bibitem[Kl2]{Kloosterman-explicit} Remke Kloosterman,
  \textit{Explicit sections on Kuwata's elliptic surfaces\/},
  Comment.\ Math.\ Univ.\ St.\ Pauli {\bf 54} (2005), no.\ 1, 69--86,
  \arXiv{math/0502017}. \MR{2153956}{(2006e:14053)}

\bibitem[Km1]{Kumar:2008} Abhinav Kumar, \textit{$K3$ surfaces
  associated with curves of genus two\/},
  Int.\ Math.\ Res.\ Not.\ IMRN \textbf{6} (2008), Art.\ ID rnm165,
  26 pp, \doi{10.1093/imrn/rnm165}. \MR{2427457}{(2009d:14044)}

\bibitem[Km2]{Kumar:2014} Abhinav Kumar, \textit{Elliptic fibrations
  on a generic Jacobian Kummer surface}, J. Algebraic Geom. {\bf 23}
  (2014), no. 4, 599--667. \doi{10.1090/S1056-3911-2014-00620-2}.
  \MR{3263663}{}
   
\bibitem[KK]{Kumar-Kuwata:singular} Abhinav Kumar and Masato Kuwata,
  \textit{Elliptic $K3$ surfaces associated with the product of two
    elliptic curves: Mordell-Weil lattices and their fields of
    definition\/}, Nagoya Math. J. {\bf 228} (2017), 124--185,
  \doi{10.1017/nmj.2016.56}. \MR{3721376}{}

\bibitem[Kw]{Kuwata:MW-rank} Masato Kuwata, \textit{Elliptic {$K3$}
  surfaces with given {M}ordell-{W}eil rank},
  Comment. Math. Univ. St. Paul. \textbf{49} (2000), no. 1, 91--100.
  \MR{1777156}{(2001h:14046)}

\bibitem[L]{Leprevost} Franck Lepre\'vost, \textit{Jacobiennes de
  certaines courbes de genre 2: torsion et simplicit\'e} (French), Les
  Dix-huitièmes Journées Arithmétiques (Bordeaux, 1993),
  J. Th\'eor. Nombres Bordeaux. {\bf 7} (1995), no. 1, 283--306,
  \doi{10.5802/jtnb.144}. \MR{1413580}{(98a:11078)}

\bibitem[Mo]{Morrison} David R.~Morrison, \textit{On $K3$ surfaces
  with large Picard number\/}, Invent.\ Math.\ {\bf 75} (1984) no.\ 1,
  105--121, \doi{10.1007/BF01403093}. \MR{728142}{(85j:14071)}

\bibitem[N]{Naruki} Isao Naruki, \textit{On metamorphosis of Kummer
  surfaces}, Hokkaido Math. J., {\bf 20} (1991), no. 2, 407--415,
  \doi{10.14492/hokmj/1381413848}. \MR{1114412}{(93a:14033)}

\bibitem[SS]{Schuett-Shioda} Matthias Sch\"utt and Tetsuji Shioda,
  \textit{Elliptic Surfaces}, Algebraic geometry in East Asia -- Seoul
  2008, Adv. Stud. Pure Math., vol. 60, Math. Soc. Japan, Tokyo, 2010,
  pp. 51--160, \arXiv{0907.0298}. \MR{2732092}{(2012b:14069)}

\bibitem[Sha]{Shaska} Tony Shaska, \textit{Genus $2$ curves with
  $(3,3)-$split Jacobian and large automorphism group}, Algorithmic
  number theory (Sydney, 2002), Lecture Notes in Comput. Sci.,
  vol. 2369, Springer, Berlin, 2002, pp. 205--218,
  \doi{10.1007/3-540-45455-1\_17}. \MR{2041085}{(2005e:14048)}

\bibitem[Shi1]{Shioda:elliptic-modular} Tetsuji Shioda, \textit{On
  elliptic modular surfaces}, J. Math. Soc. Japan \textbf{24} (1972),
  20--59, \doi{10.2969/jmsj/02410020}. \MR{0429918}{(55 #2927)}

\bibitem[Shi2]{Shioda:MWL} Tetsuji Shioda, \textit{On the Mordell-Weil
  lattices}, Comment. Math. Univ. St. Pauli \textbf{39} (1990), no. 2,
  211--240. \MR{1081832}{(91m:14056)}

\bibitem[Shi3]{Shioda:Sphere-packings} Tetsuji Shioda, \textit{A note
  on $K3$ surfaces and sphere packings}, Proc. Japan Acad. Ser. A
  Math. Sci. \textbf{76} (2000), no. 5, 68--72.
  \doi{10.3792/pjaa.76.68}.  \MR{1771143}{(2001h:14048)}

\bibitem[Shi4]{Shioda:Kummer-sandwich} Tetsuji Shioda, \textit{Kummer
  sandwich theorem of certain elliptic $K3$ surfaces}, Proc. Japan
  Acad. Ser. A Math. Sci. \textbf{82} (2006), no. 8, 137--140.
  \doi{10.3792/pjaa.82.137}. \MR{2279280}{(2008b:14064)}

\bibitem[SI]{Shioda-Inose} Tetsuji Shioda and Hiroshi Inose,
  \textit{On singular {$K3$} surfaces}, Complex analysis and algebraic
  geometry, Iwanami Shoten, Tokyo, 1977, pp. 119--136.
  \MR{0441982}{(56 #371)}

\bibitem[Ta]{Tate} John Tate, \textit{Endomorphisms of abelian
  varieties over finite fields\/}, Invent.\ Math.\ {\bf 2} (1966),
  134--144, \doi{10.1007/BF01404549}. \MR{0206004}{(34 #5829)}

\bibitem[TdZ]{Top-de Zeeuw} Jaap Top and Franck De Zeeuw,
  \textit{Explicit elliptic $K3$ surfaces with rank $15$}, Rocky
  Mountain J. Math. {\bf 39} (2009), no. 5, 1689--1697.
  \doi{10.1216/RMJ-2009-39-5-1689}. \MR{2546659}{(2010i:14065)}
\end{thebibliography}
\end{document}